\documentclass[11pt,oneside,reqno]{amsart}

\usepackage[subnum]{cases}

\usepackage{amsmath}

\usepackage{amsfonts,amssymb}
\usepackage{graphicx}
\usepackage{enumerate}
\usepackage{amsthm}
\usepackage[all]{xy}
\usepackage{extarrows}
\usepackage{color}
\usepackage{lmodern}
\usepackage{shuffle}
\usepackage{mathabx}
\usepackage{lmodern}
\usepackage{tikz}
\usepackage{hyperref}
\usepackage{rotating}
\usepackage{scalefnt}

\usepackage{a4wide}

\newcommand {\bR}{\mathbb R}
\newcommand {\bN}{\mathbb N}
\newcommand {\bZ}{\mathbb Z}
\newcommand {\bC}{\mathbb C}

\newcommand {\bQ}{\mathbb Q}

\newcommand {\Id}{\operatorname{Id}}

\newcommand {\be}{\mathbf{1}}

\newcommand{\cL}{\mathcal{L}}

\newcommand{\cA}{\mathcal{A}}
\newcommand{\cN}{\mathcal{N}}
\newcommand{\cH}{\mathcal{H}}
\newcommand{\cR}{\mathcal{R}}

\newcommand{\cP}{\mathcal{P}}

\newcommand {\RE}{\operatorname{Re}}

\newcommand {\pr}{\operatorname{pr}}

\newcommand{\fh}{\mathfrak{h}}
\newcommand{\fH}{\mathfrak{H}}

\newcommand {\z}{\zeta}

\newcommand{\sh}{\shuffle}

\newcommand{\sfS}{{\mathsf S}} 
\newcommand{\sfN}{{\mathsf{N}}} 
\newcommand{\myB}{{B}} 

\newtheorem{theorem}{Theorem}[section]
\newtheorem {lemma}[theorem]{Lemma}
\newtheorem {proposition}[theorem]{Proposition}
\newtheorem {definition}[theorem]{Definition}
\newtheorem {corollary}[theorem]{Corollary}

\newtheorem {remark}[theorem]{Remark}

\def \restr#1{\mathstrut_{\textstyle |}\raise-8pt\hbox{$\scriptstyle #1$}}
\def \srestr#1{\mathstrut_{\textstyle |}\raise-6pt\hbox{$\scriptscriptstyle #1$}}

\begin{document}

\title{Renormalisation group for multiple zeta values}

\author[K.~Ebrahimi-Fard]{Kurusch Ebrahimi-Fard}
\address{{ICMAT,C/Nicol\'as Cabrera, no.~13-15, 28049 Madrid, Spain}.{\tiny{On leave from UHA, Mulhouse, France.}}}
\email{kurusch@icmat.es, kurusch.ebrahimi-fard@uha.fr}
\urladdr{www.icmat.es/kurusch}

\author[D.~Manchon]{Dominique Manchon}
\address{Univ. Blaise Pascal, C.N.R.S.-UMR 6620, 3 place Vasar\'ely, CS 60026, 63178 Aubi\`ere, France}
\email{manchon@math.univ-bpclermont.fr}
\urladdr{http://math.univ-bpclermont.fr/~manchon/}

\author[J.~Singer]{Johannes Singer}
\address{Department Mathematik, Friedrich-Alexander-Universit\"at Erlangen-N\"urnberg, Cauerstra\ss e 11, 91058 Erlangen, Germany}
\email{singer@math.fau.de}
\urladdr{www.math.fau.de/singer}

\author[J.~Zhao]{Janqiang Zhao}
\address{ICMAT, C/Nicol\'as Cabrera, no.~13-15, 28049 Madrid, Spain}
\email{zhaoj@ihes.fr}
\urladdr{}

\maketitle


\begin{abstract}

Calculating multiple zeta values at arguments of any sign in a way that is compatible with both the quasi-shuffle product as well as meromorphic continuation, is commonly referred to as the renormalisation problem for multiple zeta values. We consider the set of all solutions to this problem and provide a framework for comparing its elements in terms of a free and transitive action of a particular subgroup of the group of characters of the quasi-shuffle Hopf algebra. In particular, this provides a transparent way of relating different solutions at non-positive values, which answers an open question in the recent literature.

\end{abstract}


\tableofcontents

\section{Introduction}
\label{sect:intro}

\emph{Multiple zeta values} (MZVs) are defined for integers $k_1 \geq 2$, $k_2,\ldots,k_n \geq 1$ in terms of nested sums
\begin{equation}
\label{mzv}
 \zeta(k_1,\ldots,k_n):=\sum_{m_1 > \cdots > m_n>0}\frac{1}{m_1^{k_1} \cdots m_n^{k_n}}
\end{equation}
of depth $n$ and weight $\sum_{i=1}^n k_i$. It is well-known that \eqref{mzv} has a representation in terms of iterated Chen integrals
\begin{equation}
\label{eq:MZVint}
         \zeta(k_1,\ldots,k_n) = \int_{0}^{1} \omega_0^{k_1-1}\omega_1 \cdots \omega_0^{k_n-1}\omega_1,
\end{equation}
where $\omega_0$ and $\omega_1$ stand for the differential 1-forms $dt/t$ and $dt/(1-t)$ respectively. Writing the nested sums \eqref{mzv} in terms of iterated integrals \eqref{eq:MZVint} yields shuffle relations satisfied by MZVs due to integration by parts. On the other hand, for integers $a,b>1$, Nielsen's reflexion formula $\zeta(a)\zeta(b)=\zeta(a,b)+\zeta(b,a)+\zeta(a+b)$ follows from multiplying sums
\begin{equation}
\label{stuf}
	\sum_{m>0}\frac{1}{m^a}\sum_{n>0}\frac{1}{n^b}
	 = \sum_{m>n>0}\frac{1}{m^an^b} + \sum_{n>m>0}\frac{1}{n^bm^a} +  \sum_{m>0}\frac{1}{m^{a+b}},
\end{equation}
and generalises naturally to quasi-shuffle relations for the nested sums \eqref{mzv}, where weight is preserved but not depth. Comparing the two ways of multiplying MZVs yields intricate relations commonly referred to as double shuffle relations \cite{Ihara06}. The latter are just the tip of an iceberg of rich and deep mathematical structures \cite{Brown12,Hoffman05,Zagier12,Zudilin03}  displayed by multiple zeta values, which materialise through profound ramifications into modern developments in mathematics. Double zeta values have been considered by L.~Euler \cite{Euler1776}. MZVs appeared in full generality in a 1981 preprint of J.~\'Ecalle, in the context of resurgence theory in complex analysis \cite{Ecalle81}. A systematic study however started only a decade later with the seminal works of D.~Zagier \cite{Zagier94} and M.~Hoffman \cite{Hoffman97}. Moreover, MZVs and their generalisations, i.e., multiple polylogarithms, appear to play an important role in quantum field theory \cite{Broadhurst13}.\\

Definition \eqref{mzv} extends to complex arguments $(s_1,\ldots, s_n)$. The \emph{multiple zeta function} thus obtained is holomorphic in the domain $\hbox{Re}(s_1+\cdots+s_j)>j$, $j=1,\ldots,n$, \cite{Krattenthaler07}. In \cite{Akiyama01a,Zhao00} it was shown that this function can be meromorphically continued to $\mathbb{C}^n$ with singularities at
\begin{align*}
 s_1 & = 1, \\
 s_1 + s_2 & = 2,1,0,-2,-4,\ldots, \\
 s_1 + \cdots + s_j & \in \bZ_{\leq j} \hbox{ ~for~}j\ge 3.
\end{align*}

In light of the well-known result on the Riemann zeta function at negative integer values, together with the formal statement
$$
	\zeta(-a)\zeta(-b)=\zeta(-a,-b)+\zeta(-b,-a)+\zeta(-a-b),
$$
obtained by imitating \eqref{stuf}, it is natural to ask the question of how to extend MZVs beyond depth one to arguments in $\mathbb Z$, such that the quasi-shuffle relations are preserved. It turns out that the solution to this so-called renormalisation problem for MZVs at integer arguments is not unique. Indeed, different answers appeared in the literature. In \cite{Guo08} L.~Guo and B.~Zhang considered the problem for non-positive arguments. For strictly negative arguments, another approach has been proposed in \cite{Ebrahimi15b}, which is based on a one-parameter family of $q$-analogs of multiple zeta values. A general solution appeared in \cite{Manchon10}, where S.~Paycha and the second author found an approach that would allow for arguments in $\mathbb Z$ (including mixed signs). Comparing values of the three solutions, $ \zeta_{\textrm{\tiny{GZ}}}$, $\zeta_{\textrm{\tiny{EMS}}}$ and $\zeta_{\textrm{\tiny{MP}}}$, at negative arguments, shows that the approaches provide different answers. For example, for the MZV $\zeta(-1,-3)$ we have the following values:
\begin{align*}
 \zeta_{\textrm{\tiny{GZ}}}(-1,-3)=\frac{83}{64512}, \hspace{0.4cm}
 \zeta_{\textrm{\tiny{EMS}}}(-1,-3)= \frac{121}{94080}, \hspace{0.4cm}
 \zeta_{\textrm{\tiny{MP}}}(-1,-3)=\frac{1}{840}.
\end{align*}
Note that $\zeta_{\textrm{\tiny{EMS}}}(-1,-3)$ corresponds to the special value $t=1$ of the family
$$
	\zeta_{\textrm{\tiny{EMS}},t}(-1,-3)= \frac{1}{8064}\frac{166t^2+166t+31}{(4t+3)(4t+1)},
$$
which is defined for $\RE(t)>0$ \cite{Ebrahimi15b}. \\

In a nutshell, all three approaches are of Hopf algebraic nature, and start from Hoffman's quasi-shuffle Hopf algebra of words \cite{Hoffman00} together with the MZV-character. The latter is a particular algebra morphism which maps \emph{convergent} words to MZVs, i.e., nested sums \eqref{mzv}. Different regularisation methods are employed to deform this MZV-character to take values in a commutative unital Rota--Baxter algebra. Each regularised MZV-character is then uniquely factorised into a product of two characters by applying a general theorem due to A.~Connes and D.~Kreimer \cite{Connes00}. It turns out that, in the three approaches \cite{Ebrahimi15b,Guo08,Manchon10}, one of the factors yields finite values for MZVs at all, non-positive and negative integer arguments, respectively.\\

The factorisation theorem of A.~Connes and D.~Kreimer captures the renormalisation problem in perturbative quantum field theory using a Hopf algebra approach to the so-called BPHZ subtraction method. In quantum field theory it is known that different regularisation methods may yield different renormalised values. The concept of finite renormalisations \cite{Zavialov90} yields a way to relate those differing renormalised values. This provides the motivation for our work. We start from Hoffman's quasi-shuffle Hopf algebra, and outline a framework to compare all possible renormalisations of the MZV-character, which are compatible with both the quasi-shuffle product as well as the meromorphic continuation. A particular group, called \emph{renormalisation group of MZVs}, is shown to act freely and transitively on renormalised MZV-characters. It is identified as the character group of a commutative Hopf algebra, which is obtained as a quotient of Hoffman's quasi-shuffle Hopf algebra by the Hopf ideal generated by \emph{non-singular words}, i.e., words associated to multiple zeta values obtainable by analytic continuation. Elements of the renormalisation group of MZVs may be considered as finite renormalisations as they permit to relate any differing values of renormalised MZVs, including those in \cite{Ebrahimi15b,Guo08,Manchon10}, by showing that they are elements of a single orbit.\\

The paper is organised as follows. After a very brief review of the analytic continuation of the multiple zeta functions and its singularities in Section 2, we define the renormalisation group  in Section 3,
study its properties and compare currently known renormalisation schemes in Section 4. The last section
is devoted to the description of the renormalisation group whose associated Hopf algebra is obtained as the quotient of the quasi-shuffle Hopf algebra by a bi-ideal obtained in terms of the non-singular integer points of the multiple zeta function.\\
\medskip

{\bf{Acknowledgements}}
The first author is supported by Ram\'on y Cajal research grant RYC-2010-06995 from the Spanish government. DM, JS and JZ would like to thank the ICMAT for its warm hospitality and gratefully acknowledge support by the Severo Ochoa Excellence Program. The second author is partially supported by Agence Nationale de la Recherche (projet CARMA). We warmly thank Hidekazu Furusho, Fr\'{e}d\'{e}ric Patras and Wadim Zudilin for very useful comments.


\section{The meromorphic continuation of multiple zeta function}
\label{sec:mero}

It is a well-known fact that the classical Riemann zeta function $\zeta(s):=\sum_{m\geq 1} m^{-s}$ is convergent for $\RE(s)>1$ and can be meromorphically continued to $\bC$ with a single pole in $s=1$. Using the functional equation
\begin{align*}
 \zeta(1-s)=\frac{1}{(2\pi)^s}\cos\left( \frac{\pi s}{2} \right)\Gamma(s)\zeta(s),
\end{align*}
where $\Gamma(s)$ denotes the meromorphic continuation of the Gamma function, we can deduce the values of the Riemann zeta function at non-positive integers, that is, for the non-negative integer $l$
  \begin{align*}
   \z(-l) = -\frac{B_{l+1}}{l+1},
  \end{align*}
where the Bernoulli numbers $B_m$ are given by the generating function
 \begin{align*}
  \frac{te^t}{e^t-1} = \sum_{m\geq0} \frac{B_m}{m!}t^m.
 \end{align*}
Therefore the meromorphic continuation provides all values of $\zeta(k)$ for $k\in \bZ\setminus \{1\}$.\\

Next we consider the multiple zeta function $\zeta(s_1,\ldots,s_n)$, which is absolutely convergent if $\sum_{j=1}^k \RE(s_j)>k$ for $k=1,\ldots,n$ \cite{Krattenthaler07}. In this domain it defines an analytic function in $n$ variables. The precise pole structure of the multiple zeta function was clarified in \cite{Akiyama01a}:

 \begin{theorem}
\label{theo:meroz}
The function $\z(s_1,\ldots,s_n)$ admits a meromorphic extension to $\bC^n$. The subvariety $\sfS_n$ of singularities is given by
\begin{align*}
 \sfS_n= \left\{(s_1,\ldots,s_n)\in \bC^n\colon  \begin{array}{c}
                                                 s_1=1;~ s_1+s_2=2,1,0,-2,-4,\ldots;  \\
                                                 s_1 + \cdots +  s_j\in \bZ_{\leq j} ~(j=3,4,\ldots,n)
                                                \end{array}
 \right\}.
\end{align*}
\end{theorem}

To illustrate the need of renormalisation procedures for MZVs let us restrict to non-positive integer arguments. As stated above, in the case of the Riemann zeta function all values at non-positive integers are prescribed by the meromorphic continuation. In length two we observe for $k_1+k_2$ odd ($k_1,k_2\in \bN_0$) that
\begin{align*}
 \z_2(-k_1,-k_2) = \frac{1}{2} \big(1+\delta_0(k_2) \big) \frac{B_{k_1+k_2+1}}{k_1+k_2+1},
\end{align*}
 whereas when the sum $k_1+k_2$ is even we do not have any information due to the existence of poles in the points $(-k_1,-k_2)$. In length greater than two there are points of indeterminacy for any non-positive integer arguments, i.e., $(\bZ_{\leq 0})^n \subseteq \sfS_n$, $n\geq 3$. Therefore the meromorphic continuation does not prescribe all values.\\

This lack of information in the context of the meromorphic continuation motivated several approaches aiming at obtaining explicit values at points of indeterminacy. The first approach is due to S.~Akiyama et al.~\cite{Akiyama01b}. It describes a limiting process to define MZVs at non-positive integer arguments. However, it turns out that the values depend on the order of conducting the limits. Additionally, the quasi-shuffle relations are not verified. 

Furthermore let us mention the unpublished work of O. Bouillot \cite{Bouillot13}, in which the author considers solely the algebraic side of the renormalisation problem of MZVs. He obtains a family of values called \emph{multiple Bernoulli numbers} for MZVs at non-positive arguments, which are all compatible with the quasi-shuffle relations. However, the problem of meromorphic continuation has not been addressed in a complete way.

H.~Furusho et al.~proposed in \cite{Furusho15} a different approach, called \emph{desingularisation of MZVs}. They showed that certain finite sums of multiple zeta functions display non-trivial cancellations of all singularities involved, which therefore provide entire functions. This yields finite values for any tuple of integer arguments. Albeit, questions concerning the algebraic properties (e.g., compatibility with the quasi-shuffle product) have not been addressed in \cite{Furusho15}.

In references \cite{Ebrahimi15b,Guo08,Manchon10} a rather different approach is explored. It is based on a fundamental factorisation theorem for algebra morphisms over connected filtered Hopf algebras, which is due to A.~Connes and D.~Kreimer \cite{Connes00}. See \cite{Manchon08} for details. As a result renormalised MZVs can be deduced, which are compatible with the quasi-shuffle product as well as the meromorphic continuation.


\section{The left group action}
\label{sect:LGA}

Let $(\cH,m,u,\Delta, \varepsilon, S)$ be a connected, filtered Hopf algebra over a field $k$ of characteristic zero. Therefore the coproduct $\Delta$ is also conilpotent. Further let $\widetilde{\Delta}$ denote the reduced coproduct defined by $\widetilde{\Delta}(x) := \Delta(x) - 1\otimes x - x\otimes 1$ for any $x\in \ker(\varepsilon)$. In Sweedler's notation we have
\begin{align*}
 \widetilde\Delta(x)=\sum_{(x)}x'\otimes x''.
\end{align*}

Let $N\subseteq \cH$ be a left coideal with respect to the reduced coproduct, i.e., $\widetilde{\Delta}(N)\subseteq N\otimes \cH$ and $\varepsilon(N)=\{0\}$. We call it the \emph{coideal of non-singular elements} by anticipating Lemma \ref {lem:coideal-N}. Moreover let $(\cA,m_\cA,u_\cA)$ be a $k$-algebra and let $\cL(\cH,\cA)$ be the vector space of linear functions from $\cH$ to $\cA$. The set $G_{\cA}$ of unital algebra morphisms from $\cH$ to $\cA$ is a group with respect to the convolution product $\star\colon \cL(\cH,\cA)\otimes \cL(\cH,\cA) \to \cL(\cH,\cA)$ defined by $ \phi \otimes \psi\mapsto  \phi \star \psi :=m_{\cA} \circ (\phi \otimes \psi)\circ \Delta$ and with unit $e:=u_\cA \circ \varepsilon$. The inverse of $\phi \in G_{\cA}$ is given by $ \phi^{-1}=\phi \circ S$. Since $\cH$ is connected and filtered we obtain the antipode by
\begin{align}
\label{eq:antipode}
 S(x) = -x - \sum_{(x)} m\big(x' \otimes S(x'')\big)
\end{align}
for any $x\in \ker(\varepsilon)$.

\begin{lemma}
The set $T_{\cA}:= \big\{\phi \in G_{\cA} \colon \phi\restr{N}=0 \big\}$ is a subgroup of $(G_{\cA},\star, e)$.
\end{lemma}

\begin{proof}
 Obviously, $e\in T_{\cA}$. Let $\phi,\psi \in T_{\cA}$. Since $G_{\cA}$ is a group $\phi \star \psi^{-1} \in G_{\cA}$. Further, for any $w\in N$, we have
  \begin{align*}
  (\phi \star \psi^{-1})(w)
  & = \big(\phi \star (\psi\circ S)\big)(w) \\
  & =  \phi(w) + \psi\big(S(w)\big) + \sum_{(w)}m_{\cA}\big(\phi(w') \otimes \psi(S(w''))\big) \\
  & = \phi(w) -\psi(w) - \sum_{(w)}m_{\cA}\big(\psi(w')\otimes \psi(S(w''))\big) + \sum_{(w)}m_\cA\big(\phi(w')\otimes \psi(S(w''))\big) = 0
 \end{align*}
because of \eqref{eq:antipode} and the fact that $N$ is a left-coideal of $\widetilde{\Delta}$ by definition.
\end{proof}

The subgroup $T_{\cA}$ will be called the \emph{renormalisation group} and its elements are called {\it{transfer characters}}. Now let $\zeta:N\to \cA$ be a partially defined character on $\cH$, i.e., a linear map such that $\zeta(\be)=\be_{\cA}$ and $\zeta(m(v \otimes w))=m_{\cA}(\zeta(v) \otimes \zeta(w))$ as long as $v$, $w$ and the product $m(v \otimes w)$ belong to $N$. We define the set of all possible renormalisations with target algebra $\cA$ by
\begin{align*}
 X_{\cA,\zeta}:= \big\{ \alpha \in G_{\cA}\colon \alpha\restr{N}= \zeta \big\}.
\end{align*}

\begin{theorem}\label{theo:groupaction}
The left group action $T_{\cA} \times X_{\cA,\zeta} \to X_{\cA,\zeta}$ defined by $(\phi,\alpha)\mapsto \phi \star \alpha$ is free and transitive.
\end{theorem}

\begin{proof}
First, we prove that the group action is well-defined. Let $\phi \in T_{\cA}$ and $\alpha \in X_{\cA,\zeta}$. For any $w\in N$ we obtain
  \begin{align*}
  (\phi \star \alpha)(w)
  	&= \phi(w) + \alpha(w)  + \sum_{(w)}m_{\cA}\big(\phi(w')\otimes \alpha(w'')\big)\\
  	& = \alpha(w) = \zeta(w),
 \end{align*}
 using that $\phi$ vanishes on the left-coideal $N$. The identity and compatibility relations of the group action are satisfied as $X_{\cA,\zeta}\subset G_{\cA}$. Freeness is obvious. In order to prove transitivity, let $\alpha,\beta \in X_{\cA,\zeta}$. Then for any $w\in N$, we have
 \begin{align*}
  (\alpha \star \beta^{-1})(w)
  & = \big(\alpha \star (\beta\circ S)\big)(w) = \alpha(w) + \beta\big(S(w)\big) + \sum_{(w)} m_\cA\big( \alpha(w') \otimes   \beta(S(w''))\big) \\
  & = \alpha(w) - \beta(w) - \sum_{(w)}  m_{\cA}\big(\beta(w') \otimes \beta(S(w''))\big)  + \sum_{(w)} m_{\cA}\big(\alpha(w') \otimes \beta(S(w''))\big) \\
  & = (\alpha-\beta)(w) + \sum_{(w)} m_{\cA} \big((\alpha-\beta)(w') \otimes \beta(S(w'')\big) = 0,
 \end{align*}
using $\beta \in G_{\cA}$ and the left-coideal property of $N$ together with $\alpha\restr{N}=\beta\restr{N}=\zeta$. Hence, $\phi:=\alpha \star \beta^{-1}\in T_{\cA}$ leads to $\alpha=\phi \star \beta$, which concludes the proof.
\end{proof}


\section{Renormalised MZVs and the renormalisation group}
\label{sect:renormMZVsRG}

In this section we exemplify the left group action described in the previous section in the context of Hoffman's quasi-shuffle Hopf algebra. The left coideal $N$ is the linear span of non-singular words defined through meromorphic continuation of the MZV-function. Finally we compare the different sets of renormalised MZVs which appeared in the literature.


\subsection{Group action in the MZV case}
\label{ssect:RGMZVs}

We quickly describe the \emph{quasi-shuffle Hopf algebra} \cite{Hoffman00} together with the set of non-singular words. Let us start with the alphabet $Y_+:=\{z_k\colon k \in \bZ_{>0}\}$ of positive letters.  Further let $\langle Y_+\rangle_\bQ$ denote the span of words with letters in $Y_+$ and let $\fh:=\bQ\langle Y_+\rangle$ be the free non-commutative $\mathbb Q$-algebra generated by $Y_+$. The empty word is denoted by $\be$. The \emph{length} of a word is defined by its number of letters. We define the \emph{quasi-shuffle product} $\ast \colon \fh \otimes \fh \to \fh$ by
\begin{enumerate}[(i)]
 \item $\be \ast v :=v\ast \be := v$,
 \item $z_m v \ast z_n w:= z_m(v \ast z_n w) + z_n(z_m v \ast w) + z_{m+n} (v \ast w)$,
\end{enumerate}
for any words $v,w \in \fh$ and integers $m,n > 0$. The unit map $u \colon \bQ \to \fh$ is defined by $u(\lambda) = \lambda \be$. The subspace $\fh^+:=\bQ\be \oplus \bigoplus_{n>1}z_n\fh$ spanned by \emph{convergent words} is a subalgebra of $\fh$. Next we define the \emph{MZV-character} $\zeta^\ast\colon (\fh^+,\ast)\to (\bR,\cdot)$ by $\z^\ast(\be):=1$ and
\begin{align}
\label{eq:zetaast}
 \zeta^\ast(z_{k_1}\cdots z_{k_n}):=\zeta(k_1,\ldots,k_n).
\end{align}

\noindent The terminology is justified by the following well-known result:
\begin{lemma}\label{lem:qscharacter}
The map $\zeta^\ast\colon (\fh^+,\ast) \to (\bR,\cdot)$ is a morphism of algebras.
\end{lemma}

We now complete the alphabet $Y_+$ to $Y:=Y_+ \cup Y_-=\{z_k\colon k \in \bZ\}$, by adding the non-positive letters $Y_-:=\{z_k\colon k \in \bZ_{\le 0}\}$. Let $\cH:=\bQ\langle Y\rangle$ be the free $\mathbb Q$-algebra generated by $Y$. The quasi-shuffle product is naturally extended, i.e., $\ast \colon \cH \otimes \cH \to \cH$, and both $\fh$ and $\fh^+$ are subalgebras of $(\cH,\ast)$. The subspace $\fH:=\langle Y_-\rangle_\bQ \subset \cH$ forms a subalgebra in $\cH$. Furthermore, we introduce another subspace $\fH^-:=\langle \widehat{Y}_-\rangle_\bQ \subset \fH$, which is defined in terms of the negative alphabet $\widehat{Y}_-:=\{z_k\colon k \in \bZ_{< 0}\}$, and which is also a subalgebra. Introducing $\cH$ is motivated by the fact that, in the sense of formal series, we can extend the definition of the map \eqref{eq:zetaast} together with Lemma \ref{lem:qscharacter} from $\fh^+$ to $\cH$, since the formal product rule of nested sums leads to the quasi-shuffle product as defined over $\cH$.\\

The quasi-shuffle algebra $(\cH,\ast)$ is a connected, filtered Hopf algebra with counit $\varepsilon :  \cH  \to \bQ$ defined by $\varepsilon(\be) =1$, and $\varepsilon(w) = 0$ for any $w\neq \be$. The coproduct  $\Delta\colon \cH \to \cH \otimes \cH$ is given for any word $w \in \cH$ by deconcatenation
\begin{align}\label{deconcat}
 \Delta(w):=\sum_{uv=w}u\otimes v.
\end{align}
The antipode $S: \cH \to \cH$ is given by the general formula \eqref{eq:antipode}. Note that both subalgebras $\fH^- \subset \fH \subset \cH$ are in fact Hopf subalgebras. Following \cite{Hoffman00} the antipode $S$ can be given by a non-recursive description in terms of {\it{word contractions}}, i.e., for any word $w=z_{k_1}\cdots z_{k_n} \in \cH$ with $n$ letters
\begin{align*}
	S(w) = (-1)^n \sum_{I \in \cP(n)} I[\overline{w}].
\end{align*}
Here $\overline{w}:=z_{k_n}\cdots z_{k_1}$, and $\cP(n)$ is the set of compositions of the integer $n$, i.e., the set of sequences  $I:=(i_1,\ldots,i_m)$ of positive integers such that $i_1+\cdots +i_m=n$. Then for any word $w=z_{k_1}\cdots z_{k_n}$ and any composition  $I=(i_1,\ldots,i_m)$ of $n$ we define the \emph{contracted word} by
\begin{align*}
 I[w]:=z_{k_1+\cdots +k_{i_1}}z_{k_{i_1+1}+\cdots +k_{i_1+i_2}} \cdots z_{k_{n-i_m+1}+\cdots +k_n}.
\end{align*}

\noindent The notion of contracting words underlies also \textsl{Hoffman's exponential} \cite{Hoffman00}:
\begin{align}
\label{HoffExp}
	\exp_H(u)&:=\sum_{I=(i_1,\ldots ,i_k)\in\cP(n)}\frac{1}{i_1! \cdots i_k!}I[u].
\end{align}
It defines a Hopf algebra isomorphism from the shuffle Hopf algebra ${\cH}_\sh$ onto the quasi-shuffle Hopf algebra ${\cH}$. The former is defined as the free $\mathbb Q$-algebra $\bQ\langle Y\rangle$ generated by $Y$ and equipped with the shuffle product
\begin{enumerate}[(i)]
 \item $\be \sh v :=v \sh \be := v$,
 \item $z_m v \sh z_n w:= z_m(v \sh z_n w) + z_n(z_m v \sh w)$,
\end{enumerate}
for any $v,w\in \bQ\langle Y\rangle$ and $m,n\in \bZ$. The inverse $\log_H$ of $\exp_H$ is given by~ (\cite[Lemma 2.4]{Hoffman00}):
\begin{align*}
	\log_H(u)=\sum_{I=(i_1,\ldots ,i_k)\in\cP(n)}\frac{(-1)^{n-k}}{i_1\cdots i_k}I[u].
\end{align*}
For example for  $z_{k_1},z_{k_2} \in Y$ we have that at length one $\exp_H z_{k_1} =z_{k_1}$ and $\log_H z_{k_1} =z_{k_1}$. For words of length two
\begin{align*}
\exp_H (z_{k_1}z_{k_2}) =z_{k_1}z_{k_2}+\frac{1}{2}z_{k_1+k_2},\quad
\log_H (z_{k_1}z_{k_2}) =z_{k_1}z_{k_2}-\frac{1}{2}z_{k_1+k_2}.
\end{align*}

The group of $\bC$-valued characters over $\cH$ is denoted by $G_\mathbb{C}$. Its corresponding Lie algebra of infinitesimal $\bC$-valued characters is denoted by $g_\mathbb{C}$. The group of $\bC$-valued characters over the Hopf algebra $\fH^-$ is denoted by $G^-_\mathbb{C}$ and its Lie algebra of infinitesimal $\bC$-valued characters is $g^-_\mathbb{C}$.
\goodbreak
\begin{definition}[Non-singular words]\label{def:non-singular}
~
\begin{enumerate}[a)]
 \item A word $w:=z_{k_1}\cdots z_{k_n}$ with letters from the alphabet $Y$ will be called \emph{non-singular} if all of the following conditions are verified:
\begin{itemize}
 \item $k_1\neq 1$,
 \item $k_1+k_2\notin\{2,1,0,-2,-4,\ldots\}$,
 \item $k_1+\cdots+k_j\notin \mathbb Z_{\le j}$ for $j\ge 3$.
\end{itemize}
\item The vector subspace spanned by non-singular words is denoted by $N\subseteq \cH$.
\end{enumerate}
\end{definition}

The vector space $N$ is naturally graded by length, i.e., $N=\bigoplus_{l\ge 1} N_l$, where $N_n$ denotes the space of $\bQ$-linear combinations of non-singular words of length $n$. Note that $N$ includes the set of convergent words.\\

\noindent The partially defined character $\zeta^\ast \colon N\to \bC$ for any word $w=z_{k_1}\cdots z_{k_n}\in N$ is given by
\begin{align*}
 \zeta^\ast(w):=\zeta(k_1,\ldots,k_n),
\end{align*}
which is either convergent or can be defined by analytic continuation (see Section \ref{sec:mero}), and linearly extended to $N$.

\begin{lemma}\label{lem:coideal-N}
The vector space $N$ is a left coideal for the reduced deconcatenation coproduct $\widetilde{\Delta}$. Moreover, $N$ is invariant under contractions, i.e., for any word $w \in N$ we have $I[w]\in N$.
\end{lemma}

\begin{proof}
Let $w:=z_{k_1}\cdots z_{k_n}$ be a word in $N$. Then it satisfies the conditions of Definition \ref{def:non-singular} a) in the length $n$ case. For any $m\in \{1,\ldots,n\}$ the subword $w':=z_{k_1}\cdots z_{k_m}$ by construction satisfies the conditions of Definition \ref{def:non-singular} a) in the case of length $m$ as well. Hence, $\widetilde{\Delta}(N)\subseteq N\otimes \cH$. The stability of $N$ under contractions is immediate from Definition \ref{def:non-singular}.
\end{proof}

\begin{corollary}
 The vector space $N$ is a two-sided coideal for the deconcatenation coproduct $\Delta$ defined in \eqref{deconcat}.
\end{corollary}

A well-defined {\it{renormalised MZV-character}} must be both compatible with the quasi-shuffle product as well as the meromorphic continuation of MZVs (whenever the latter is defined).

\begin{definition}[Renormalised MZVs]\label{def:renMZVs}
The set of all possible renormalisations of MZVs is defined by
\begin{align}
\label{renMZVs}
 X_{\bC,\zeta^\ast}:= \big\{ \alpha \in G_{\bC}\colon \alpha\restr{N}= \zeta^\ast \big\}.
\end{align}
\end{definition}

\begin{proposition}\label{prop:non-empty}
The set $X_{\bC,\zeta^\ast}$ is not empty.
\end{proposition}

\begin{proof}
Indeed, the renormalised MZV-character of \cite[Theorem 8]{Manchon10} lies in $X_{\bC,\zeta^\ast}$.
\end{proof}

Further, we will show below that $X_{\bC,\zeta^\ast}$ in fact contains infinitely many elements. Following Theorem \ref{theo:groupaction} the {\it{renormalisation group}}
\begin{align*}
 T_{\bC}:= \big\{\phi \in G_{\bC} \colon \phi\restr{N}=0 \big\}
\end{align*}
acts transitively and freely from the left on the set $X_{\bC,\zeta^\ast}$ of renormalised MZV-characters. Let $N^-:= N \cap \fH^-$ and let $G_\bC^-$ be the set of unital algebra morphisms from $ \fH^- \to \bC$. Then the set
\begin{align}
\label{renMZVs2}
	X^{-}_{\bC,\zeta^\ast} := \big\{ \alpha \in G^-_{\bC}\colon \alpha\restr{N^-}= \zeta^\ast \big\}
\end{align}
contains, for instance, the renormalised MZV-characters constructed in \cite{Ebrahimi15b}, but also the renormalised MZV-characters of \cite{Guo08,Manchon10} when appropriately restricted to $\fH^-$. Its renormalisation group $T^-_{\bC}:= \big\{\phi \in G^-_{\bC} \colon \phi\restr{N^-}=0 \big\}$ acts transitively and freely on $X^{-}_{\bC,\zeta^\ast}$.


\subsection{Comparison of different renormalisations of MZVs}
\label{ssect:models}

The central aim of our work is to understand the relation between different solutions of the renormalisation problem for MZVs, including those given in \cite{Ebrahimi15b,Guo08,Manchon10}. Recall that the set $X_{\bC,\zeta^\ast}$ comprises all solutions to this problem, and contains -- among others -- the renormalised MZV-character given in \cite{Manchon10}. It is also clear that the set $X^{-}_{\bC,\zeta^\ast}$ contains the solutions in \cite{Ebrahimi15b} as well as appropriate restrictions to $\fH^-$ of solutions presented in \cite{Guo08,Manchon10}. The relation between the aforementioned renormalised MZV-characters is immediately provided by the transitive and free left action of the renormalisation group $T^-_{\bC}$, that is, they all lie on a single orbit.\\

To see the origin of the different values, we briefly recall the essential steps leading to the solutions of the renormalisation problem for MZVs in \cite{Ebrahimi15b,Guo08,Manchon10}. All three approaches apply a key theorem from \cite{Connes00} in the context of Hoffman's quasi-shuffle Hopf algebra.

\begin{theorem}[\cite{Connes00}]
\label{theo:ConKre}
Let $\cH$ be a graded or filtered Hopf algebra over a ground field $k$, and  let $\cA$ a commutative unital $k$-algebra equipped with a renormalisation scheme $\cA=\cA_{-}\oplus \cA_{+}$ and the corresponding idempotent Rota--Baxter operator $\pi$, where $\cA_{-}=\pi(\cA)$ and $\cA_{+}=(\Id-\pi)(\cA)$. The character $\psi\colon \cH \to \cA$ admits a unique {\it{Birkhoff}} decomposition
  \begin{align}
   \label{eq:birk}
   \psi_{-}\star\psi= \psi_{+},
  \end{align}
where $\psi_{-}\colon \cH \to k\be \oplus \cA_{-}$ and $\psi_{+}\colon \cH \to \cA_{+}$ are characters.
\end{theorem}

In the context of MZVs the starting point is the $\mathbb{C}$-valued MZV-character \eqref{eq:zetaast}, which a priori is well-defined on the subalgebra $\fh^+$ of convergent words. It is then extended to $N$ by meromorphic continuation, which makes it a partially defined character. The critical step in this approach is the method of regularising the partially defined  MZV-character in such a way that the resulting map is compatible with the quasi-shuffle product. This then defines a new algebra morphism on all of $\cH$ with values in a commutative unital Rota--Baxter algebra. In \cite{Ebrahimi15b,Guo08,Manchon10} these new characters are constructed by introducing a so-called {\it{regularisation parameter}} $\lambda$ such that the regularised MZV-character $\zeta^*_\lambda$ maps $\cH$ into the Laurent series $\cA:=\mathbb{C}[\lambda^{-1},\lambda]\!]$. The latter is a commutative Rota--Baxter algebra, where $\cA_{-}:=\lambda^{-1}\bC[\lambda^{-1}]$ and $\cA_{+}:=\bC[\![\lambda]\!]$. On $\cA$ the corresponding projector $\pi: \cA \to \cA_{-}$ is defined by {\it{minimal subtraction}}
\begin{align*}
 \pi\left(\sum_{n=-l}^\infty a_n \lambda^n \right):= \sum_{n=-l}^{-1}a_n \lambda^n
\end{align*}
with the common convention that the sum over the empty set is zero. Note that applying \eqref{eq:birk} in this setting, the factor $\zeta^*_{\lambda,+}$ in the Birkhoff decomposition of $\zeta^*_\lambda = \zeta^{* -1}_{\lambda,-} \star \zeta^*_{\lambda,+}$  is well-defined when putting the regularisation parameter $\lambda$ to zero, and provides a solution to the renormalisation problem for MZVs. In the case of convergent words, the basic building block is the map that associates with each letter $z_l$ the function $f_l:x\mapsto x^{-l}$ for $x\in[1,+\infty)$. A {\it{regularisation process}} consists of a map $\cR$, that deforms $f_l(x)$ appropriately. The three regularised characters in \cite{Ebrahimi15b,Guo08,Manchon10} are given by deforming this map as follows

\begin{itemize}
 \item \cite{Guo08}: $z_{l} \mapsto f_{\lambda,l}(x):=\frac{\exp(-l x \lambda)}{x^l}$

 \item \cite{Ebrahimi15b}: $z_l \mapsto f_{\lambda,l}(x):= \frac{q^{|l| x}}{(1-q^x)^l}$, $\lambda = \log(q)$

 \item \cite{Manchon10}: $z_l \mapsto f_{\lambda,l}(x):=\frac{1}{x^{l-\lambda}}$

\end{itemize}

\noindent The nested sum \eqref{mzv} changes accordingly in each case
\begin{equation}
\label{mzv-reg}
 \zeta_\lambda (k_1,\ldots,k_n):=\sum_{m_1>\cdots > m_n>0} f_{\lambda,k_1}(m_1) \cdots f_{\lambda,k_n} (m_n) \in \cA.
\end{equation}
The finite characters are then defined in terms of the following $\mathbb{C}$-valued maps
\begin{align*}
	\zeta^+_{\textrm{\tiny{GZ}}}\restr{\lambda=0}= \zeta_{\textrm{\tiny{GZ}}},
	\hspace{1.5cm}
	\zeta^+_{\textrm{\tiny{EMS}}}\restr{\lambda=0}=\zeta_{\textrm{\tiny{EMS}}},
	\hspace{1.5cm}
	\zeta^+_{\textrm{\tiny{MP}}}\restr{\lambda=0}=\zeta_{\textrm{\tiny{MP}}},
\end{align*}
such that for words from $Y_-$, $\widehat{Y}_-$, and $Y$, respectively, the resulting maps are compatible with the meromorphic continuation. We remark that the most general approach is that presented in \cite{Manchon10} which provides us with the key example of an element in $X_{\bC,\zeta^\ast}$.

\section{The renormalisation group of MZVs in the mixed-sign case}
\label{sect:RGmixed-sign}

In this section we show that the renormalisation group $T_\bC$ is a pro-unipotent group with infinite-dimensional Lie algebra. Since  $X_{\bC,\zeta^\ast}$ is not empty  (due to \cite[Theorem 8]{Manchon10}) and the group action is transitive and free, the choice of any particular element in $X_{\bC,\zeta^\ast}$ yields a bijection from the renormalisation group onto the latter. In particular, $X_{\bC,\zeta^\ast}$ is of infinite uncountable cardinality.

\begin{proposition}
We have:
\begin{enumerate}[a)]
 \item The two-sided ideal $\cN$ generated by $N$ is a Hopf ideal of $\cH$.
 \item For any commutative unital algebra $\cA$, the renormalisation group $T_{\cA}$ is isomorphic to the group of characters of the Hopf algebra $\cH/\cN$.
\end{enumerate}
\end{proposition}

\begin{proof}
Since $N$ is a two-sided coideal for the deconcatenation coproduct $\Delta$, the claim in $a$) is clear. To prove $b$), we consider  the map
\begin{align*}
 \Phi\colon T_{\cA} \to \operatorname{Char}(\cH/\cN), \hspace{0.5cm}\psi \mapsto \big(\Phi(\psi):[x]\mapsto \psi(x)\big).
\end{align*}
The map $\Phi$ is well defined and it is an algebra morphism. Furthermore $\Phi$ is surjective since any character $\psi$ of $\cH/\cN$ gives rise to a character $\widetilde{\psi}:=\psi \circ \pr$ of $\cH$ using the canonical projection $\pr \colon \cH\to \cH/\cN$, $x\mapsto [x]$, which vanishes on $\cN$. Injectivity of $\Phi$ is clear.
\end{proof}

\begin{theorem}\label{theo:group}
 The renormalisation group $T_{\cA}$ is pro-unipotent and can be identified with the space
 $\cL(W,\cA)$ of linear maps from $W$ to $\cA$, where $W$ is some vector subspace of $\cH/\cN$.
\end{theorem}

\begin{proof}
 The deconcatenation coproduct \eqref{deconcat} is conilpotent, i.e., for any word $w$ we have $\widetilde\Delta^k(w)=0$ for sufficiently large $k$, where $\widetilde \Delta^k:=(\Id \otimes \widetilde \Delta^{k-1})\circ \widetilde \Delta$ is the $k$-th iterated reduced coproduct. The same property holds for the reduced coproduct of $\cH/\cN$. Hence, by a theorem of D.~Quillen \cite{Quillen69} (see \cite[Theorem 3.9.1]{Cartier07}) the Hopf algebra $\cH$, respectively $\cH/\cN$, is isomorphic to the free commutative algebra on the vector subspace $\overline W:=\pi_1(\cH)$, respectively $W= \widetilde{\pi}_1 (\cH/\cN)$, where $\pi_1:=\log^\star \hbox{Id}$ denotes the Eulerian idempotent with respect to the convolution product of the Hopf algebra $\cH$ and  $\widetilde{\pi}_1:=\log^{\widetilde{\star}} \hbox{Id}$ is defined with respect to the convolution product of the quotient Hopf algebra $\cH/\cN$ (\cite{Patras91}, see also \cite{Loday94}). By definition the following diagram commutes:
\vskip 2mm
\centerline{
\xymatrix{
{\cH} \ar[r]^{\pi_1}\ar@{->>}^{\pr}[d] & \overline W\ar@{->>}^{\pr\srestr{\overline W}}[d]\\
{\cH}/{\cN} \ar[r]^{\widetilde{\pi}_1} & W
}
}
\vskip 2mm
We have $W=\overline W/ \pi_1(\cN)$. Indeed, for any $x\in \cH$ we have that $\pr(\pi_1(x)) = 0$ implies $\pi_1(x)\in \cN$ and therefore $\pi_1(x)\in \pi_1(\cN)$ since $\pi_1$ is a projector. Hence, $\ker(\pr\restr{\overline W})=\pi_1(\cN)$ and the claim follows from the fundamental theorem on homomorphisms. Now any $\cA$-valued character of $\cH/\cN$ is uniquely determined, through multiplicative extension, by a linear map from $W$ to $\cA$.
\end{proof}

The pro-nilpotent Lie algebra $\mathfrak t_{\cA}$ of the pro-unipotent renormalisation group $T_{\cA}$ is the space of $\cA$-valued infinitesimal characters of $\cH/\cN$, which identifies itself with $\cL(W,\cA)$.

\begin{theorem}
The Lie algebra $\mathfrak t_{\cA}$ of the renormalisation group is infinite-dimensional.
\end{theorem}

\begin{proof}
It will be convenient to use the shuffle instead of the quasi-shuffle product. Therefore let $\cH_\sh$ be the Hopf algebra $\bQ\langle Y\rangle$ endowed with the shuffle product $\shuffle$ introduced in Paragraph \ref{ssect:RGMZVs}. By Hoffman's logarithm $\log_H$ -- which sends any word to a suitable linear combination of shorter words obtained by contractions -- the Hopf algebras $\cH$ and $\cH_\sh$ are isomorphic. The Hopf algebra $\cH_\sh$ is the free commutative algebra over the free Lie algebra $\hbox{Lie}(Y)$ generated by the alphabet $Y$. Hence, $\hbox{Lie}(Y)$ is linearly isomorphic to $\overline{W'}=\pi_1(\cH_\sh)$, where $\pi_1$ now stands for the Eulerian idempotent of $\cH_\sh$.\\

The pairing $\langle - , - \rangle\colon \mathbb Q\langle Y\rangle \otimes \mathbb Q\langle Y\rangle \to \bQ$ defined by $\langle w,w'\rangle =\delta_w^{w'}$ for any $w,w'\in Y^*$ realises a non-degenerated Hopf pairing between $\cH_\sh=(\mathbb Q\langle Y\rangle,\sh,\Delta)$ and $\cH^\lor:=(\mathbb Q\langle Y\rangle,.,\Delta_\sh)$, where $\delta_w^{w'}$ is the Kronecker symbol. Here $.$ denotes the concatenation product and $\Delta_\sh$ is the cocommutative deshuffle coproduct, i.e., the only coproduct compatible with concatenation, and such that the letters are primitive. By looking at the two convolution products in $\hbox{End}(\cH_\sh)$ and $\hbox{End}(\cH)$, it is easily seen that the Eulerian idempotent of $\cH^\lor$ is the transpose $\pi_1^\lor$ of the Eulerian idempotent $\pi_1$ of $\cH_\sh$. The map $\pi_1^\lor$ takes values in the free Lie algebra $\operatorname{Lie}(Y)$, and we have
\begin{equation}
\pi_1^\lor(z_1\cdots z_n)=\sum_{\sigma\in S_n}c(\sigma)z_{\sigma(1)}\cdots z_{\sigma(n)},
\end{equation}
where $S_n$ is the symmetry group of degree $n$ and
the coefficients $c(\sigma)$ are explicitly given by
\begin{equation}
c(\sigma)=\frac{(-1)^{d(\sigma)}}{n}{n-1\choose d(\sigma)}^{-1},
\end{equation}
where $d(\sigma)\in\{0,\ldots, n-1\}$ is the number of descents of $\sigma$, i.e., the cardinality of $D(\sigma):=\big\{j\in \{1,\ldots,n\}\colon \sigma(j+1)<\sigma(j) \big\}$ \cite{Solomon68}. The duality formula $\langle \pi_1(w),w'\rangle = \langle w,\pi_1^\lor(w')\rangle$ immediately yields
\begin{equation}
\pi_1(z_1\cdots z_n)=\sum_{\sigma\in S_n}c(\sigma^{-1})z_{\sigma(1)}\cdots z_{\sigma(n)}.
\end{equation}
Both maps $\pi_1$ and $\pi_1^\lor$ coincide on words of lengths $1$, $2$ and $3$, due to the fact that $c(\sigma)=c(\sigma^{-1})$ for any $\sigma\in S_n$ for $n\in\{1,2,3\}$. However, note that this is no longer true in length $4$, and that $\pi_1(\mathcal H_\sh)$ is not included in $\hbox{Lie}(Y)$ (see Remark \ref{rem:lie}).


Our Hopf algebra $\cH/\cN$ is isomorphic to $\cH_\sh/\cN_{\sh}$, where $\cN_{\sh}=\log_H(\cN)$ is the ideal generated by $N_{\sh}=\log_H(N)$. Note that the second assertion of Lemma
\ref{lem:coideal-N} leads to $N_{\sh}=N$ since $\log_H(N) = N$.
Similar as in the proof of Theorem \ref{theo:group} we have $\cH_\sh/\cN_\sh=S(W')$ where
\begin{align}\label{eq:Wprime}
	W'=\overline{W'}/\pi_1(N).
\end{align}
Therefore the following diagram commutes
\vskip 2mm
\centerline{
\xymatrix{
{\cH_\sh} \ar[rr]^{\pi_1}\ar@{->>}[d] && \overline{W'}\ar@{->>}[d]\\
{\cH_\sh}/{\cN_\sh} \ar[rr]^{\widetilde{\pi}_1} && W'
}
}
\vskip 2mm
\noindent where $\widetilde{\pi}_1$ is the Eulerian idempotent of $\cH_\sh/\cN_\sh$.\\

The natural grading on $\cH_\sh$ by depth generates a grading on $W'$. The following proposition shows that $W'$ is an infinite dimensional vector spaces which  concludes the proof of the theorem.
\end{proof}

\begin{proposition}
We have:
 \begin{enumerate}
  \item [\upshape{(i)}] $W'_1$ is one-dimensional.
  \item [\upshape{(ii)}] A basis of $W'_2$ can be identified with
  \begin{align*}
   \{z_a z_b\colon  a>b, (a,b)\in \sfS_2\},
  \end{align*}
  \item [\upshape{(iii)}] Let $\sfN_3:=\bZ^3\setminus \sfS_3$ be the set of non-singular points for the triple zeta function.
A basis of $W'_3$ can be identified with $\myB:=\bigcup_{j=1}^4 \myB_j$, where
\begin{align*}
&  \myB_1:=\{z_a z_b z_c, z_a z_c z_b: a>b>c, (a,b,c)\in \sfS_3 \}, \\
&  \myB_2:=\{z_a z_c z_b: a>b>c, (a,b,c)\in \sfN_3, (a,c,b),(b,a,c)\in \sfS_3\}, \\
&  \myB_3:=\{z_a z_a z_b: a>b, (a,a,b)\in \sfS_3\},\\
&  \myB_4:=\{z_a z_b z_b: a>b, (a,b,b)\in \sfS_3\}.
\end{align*}
 \end{enumerate}
\end{proposition}

\begin{proof}
Part (i) and Part (ii) are clear. We now turn to Part (iii). Let $a,b,c$ be any three integers. Direct computation yields
\begin{equation}\label{equ:depth3Rels}
\begin{split}
& 6\pi_1(z_a z_b z_c)=6\pi_1(z_c z_b z_a)=[z_a,z_b,z_c]-[z_b,z_c,z_a]= [z_a,z_c,z_b]-2[z_b,z_c,z_a],\\
&  6\pi_1(z_a z_c z_b)=6\pi_1(z_b z_c z_a)=[z_a,z_c,z_b]-[z_c,z_b,z_a]= [z_a,z_c,z_b]+[z_b,z_c,z_a],\\
&  6\pi_1(z_b z_a z_c)=6\pi_1(z_c z_a z_b)=[z_b,z_a,z_c]-[z_a,z_c,z_b]=-2[z_a,z_c,z_b]+[z_b,z_c,z_a],
\end{split}
\end{equation}
where $[z_a,z_b,z_c]:=\big[[z_a,z_b],z_c\big]$. By \eqref{equ:depth3Rels} the dimension of
the space generated by $\pi_1(\sigma(z_az_bz_c))$ for
all permutations $\sigma$ of the three letters $z_a,z_b$ and $z_c$ is at most two.
If $(a,b,c)\in\sfS_3$ and by exchanging
two adjacent descending components we get a non-singular point
then this must be caused either by one component being equal to 1 or by the special singularity
condition in depth two depending on the parity of the sum of the first
two components. In either case, it can be proved by an easy combinatorial
argument that $\pi_1(\sigma(z_az_bz_c))\in \pi_1(N_3)$ for all $\sigma$,
so these words give no contribution to $W'$.\\

Now we assume $a>b>c$. We can get all permutations of $(a,b,c)$
by repeatedly exchanging two adjacent descending components.
If $(a,b,c)\in\sfS_3$ then by the
above paragraph we may assume all of its permutations are singular.
Therefore we obtain the basis $\myB_1$.
If $(a,b,c)\in \sfN_3$, $(b,a,c), (a,c,b)\in \sfS_3$ then by the
first paragraph again we may assume
all other words are singular, leading to $\myB_2$.
On the other hand, if $(a,b,c),(b,a,c)\in \sfN_3$ then by \eqref{equ:depth3Rels} we have
\begin{equation*}
   \pi_1(z_az_cz_b)=-\pi_1(z_az_bz_c)-\pi_1(z_bz_az_c)\in \pi_1(N_3).
\end{equation*}
Similarly, $(a,b,c),(a,c,b)\in \sfN_3$ implies that $\pi_1(z_bz_az_c)\in \pi_1(N_3)$.
The argument for singular words with repeating letters is similar and thus is left
to the interested reader.
\end{proof}

\begin{remark}\label{rem:lie}
In length $n\geq 4$ it is no longer true that $\pi_1=\pi_1^\lor$. Considering $\sigma=(2413)$ together with its inverse $\sigma^{-1}=(3142)$, we have $d(\sigma)=1$ and $d(\sigma^{-1})=2$, hence, $c(\sigma)=-1/12$ and $c(\sigma^{-1})=1/12$. Any permutation $\tau\in S_4$ different from $\sigma$ and $\sigma^{-1}$ verifies $c(\tau)=c(\tau^{-1})$. Hence, we get for any $z_1,z_2,z_3,z_4\in Y$:
$$\pi_1(z_1z_2z_3z_4)=\pi_1^\lor(z_1z_2z_3z_4)+\frac{1}{6}(z_2z_4z_1z_3-z_3z_1z_4z_2).$$
This clearly shows that $\pi_1(z_1z_2z_3z_4)$ does not belong to $\hbox{Lie}(Y)$.
\end{remark}

Solutions to the renormalisation problem of MZVs are compatible with both the quasi-shuffle relations as well as the meromorphic continuation of MZVs. However, these two requirements are not sufficient to obtain a unique solution. We have shown that all solutions are comparable in terms of a free and transitive group action. The vector space $W'$ introduced in \eqref{eq:Wprime} permits a precise characterisation of the degrees of freedom, which are left unspecified by the renormalisation of MZVs. Therefore it is natural to ask for an explicit description of $W'$ via a basis. In the final part of our paper we have calculated a basis up to degree three. Remark \ref{rem:lie} indicates that going beyond this order in the construction of a basis of $W'$ is a rather complicated combinatorial problem involving methods from free Lie algebra theory.

\bibliographystyle{abbrv}
\bibliography{library}

\end{document}